\documentclass[11pt,a4paper]{article}
\usepackage[utf8]{inputenc}
\usepackage{xcolor,amsmath}
\usepackage{amsthm}
\usepackage{amsfonts}
\usepackage{amssymb}
\usepackage{cite}
\usepackage{relsize}
\usepackage{sectsty}
\sectionfont{\fontsize{12}{15}\selectfont}
\newtheorem{lemma}{Lemma}
\newtheorem{theorem}{Theorem}
\newtheorem*{corollary*}{Corollary}

\title{On the meromorphic continuation of Beatty Zeta-functions and Sturmian Dirichlet series}
\author{Athanasios Sourmelidis}
\begin{document}
\maketitle
\begin{abstract}
\noindent
For a positive irrational number $\alpha,$ we study the ordinary Dirichlet series $\zeta_\alpha(s)=\sum\limits_{n\geq1}\lfloor\alpha n\rfloor^{-s}$ and $S_\alpha(s)=\sum\limits_{n\geq1}(\left\lceil\alpha n\right\rceil-\left\lceil \alpha (n-1)\right\rceil){n^{-s}}.$ We prove relations between them and $J_{\boldsymbol{\alpha}}(s)=\sum\limits_{n\geq1}\left(\lbrace\alpha n\rbrace-\frac{1}{2}\right)n^{-s}.$ Motivated by the previous work of Hardy and Littlewood, Hecke and others regarding $J_{\boldsymbol{\alpha}},$ we show that $\zeta_\alpha$ and $S_\alpha$ can be continued analytically beyond the imaginary axis except for a simple pole at $s=1.$ Based on the latter results, we also prove that the series $\zeta_{\alpha}(s;\beta)=\sum\limits_{n\geq0}\left(\lfloor\alpha n\rfloor+\beta\right)^{-s}$ can be continued analytically beyond the imaginary axis except for a simple pole at $s=1.$
\end{abstract}
\section{Introduction and Main Results}
One of the standard problems concerning Dirichlet series is whether there exists a meromorphic continuation of them beyond their half-plane of convergence. Because of that, in addition to the abscissa of convergence $\sigma_c(f)$ and absolute convergence $\sigma_a(f)$ of a Dirichlet series $f,$ one also considers the furthest left abscissa $\sigma_{\mu}(f)$ of the vertical half-plane in which $f$ has a meromorphic continuation. In the case that $\sigma_\mu(f)$ is a finite number, the vertical line $\sigma_{\mu}(f)+i\mathbb{R}$ is called the natural boundary of $f.$  

 The basic model is the Riemann Zeta-function
\begin{align*}
\zeta(s)=\mathlarger\sum\limits_{n=1}^{\infty}\dfrac{1}{n^s}.
\end{align*}
It is well-known that $\sigma_c(\zeta)=\sigma_{a}(\zeta)=1$ and $\sigma_{\mu}(\zeta)=-\infty$ with a simple pole at $s=1.$

Of particular interest are Dirichlet series which are associated with problems of Diophantine approximation. For a positive irrational number $\alpha$ of type $\tau_\alpha$ (see Section \ref{Diop.appr.}) and a non-zero integer $q$, define
 \begin{align*}A_{\alpha q}:=\lbrace n\in\mathbb{N}:\lbrace \alpha n\rbrace<\lbrace \alpha q\rbrace\rbrace
 \end{align*}
and
\begin{align}\label{Hecke}
g_{\alpha q}(s):=\mathlarger\sum\limits_{n\in A_{\alpha q}}\dfrac{1}{n^s}.
\end{align} 
The latter series converges for $s\in\mathbb{C}_1,$ where 
\begin{align*}
\mathbb{C}_x:=\lbrace s\in\mathbb{C}:\mathrm{Re}(s)>x\rbrace
\end{align*} 
for $x\in[-\infty,+\infty],$
 and it was first considered by Hecke \cite{hecke1922analytische}. He proved for quadratic irrational $\alpha>0$, that  $g_{\alpha q}$ is meromorphically continued to the whole complex plane and its poles lie on a lattice of points depending on $\eta_\alpha$ (see Section \ref{Diop.appr.}). Firstly, we will prove a generalization of his result.
\begin{theorem}\label{gaq} For every irrational $\alpha>0$  and $q\in\mathbb{Z}^*,$ the Dirichlet series $g_{\alpha q}$ defined in (\ref{Hecke}) can be continued analytically to $\mathbb{C}_{-\frac{1}{\tau_{\alpha}}}$ except for a simple pole at $s=1$ with residue $\lbrace\alpha q\rbrace.$ If $\tau_\alpha>1,$ the vertical line $-\dfrac{1}{\tau_\alpha}+i\mathbb{R}$ is the natural boundary for $g_{\alpha q}.$
 \newline
 In the special case where $\alpha$ is a positive quadratic irrational, $g_{\alpha q}$ has a meromorphic continuation to the whole complex plane with a simple pole at $s=1$ and at most simple poles at $s=-k\pm\dfrac{2\pi i n}{\log\eta_\alpha},$ $(k,n)\in\mathbb{N}\times\mathbb{N}_0.$
\end{theorem}
Our aim is to study two more Dirichlet series which are, ultimately, related to $g_{\alpha q}.$ Let
\begin{align*}
B(\alpha)=\lbrace\lfloor\alpha n\rfloor:n\in\mathbb{N}\rbrace\hspace*{0.5cm}\text{and}\hspace*{0.5cm}
S(\alpha)=\left\{\left\lceil\beta n\right\rceil-\left\lceil\beta(n-1)\right\rceil:n\in\mathbb{N}\right\},
\end{align*}
be the homogeneous Beatty sequence and the Sturmian sequence associated with $\alpha$, respectively, where $\alpha$ is a real number greater than 1 and $\beta$ the reciprocal of $\alpha.$  Denote by
\begin{align*}
\zeta_\alpha(s):=\mathlarger\sum\limits_{n=1}^\infty\dfrac{1}{\lfloor\alpha n\rfloor^s}\text{\hspace*{0.5cm}and\hspace*{0.5cm}}S_\alpha(s):=\mathlarger\sum\limits_{n=1}^{\infty}\dfrac{\left\lceil\beta n\right\rceil-\left\lceil\beta(n-1)\right\rceil}{n^s},
\end{align*}
both converging for $s\in\mathbb{C}_1,$ the corresponding Beatty Zeta-function and Sturmian Dirichlet series. Here  $\lfloor x\rfloor$ is the largest integer which is smaller than or equal to the real number $x,$ while $\lceil x\rceil$ is the smallest integer which is greater than or equal to $x.$

From Theorem \ref{gaq}, we will derive
\begin{theorem}\label{BS}For every irrational $\alpha>1$ the Dirichlet series $\zeta_\alpha$ and $S_\alpha$ can be continued analytically to $\mathbb{C}_{-\frac{1}{\tau_{\alpha}}}$ except for a simple pole at $s=1$ with residue $\beta.$ If $\tau_\alpha>1,$ the vertical line $-\dfrac{1}{\tau_\alpha}+i\mathbb{R}$ is the natural boundary for $\zeta_\alpha$ and $S_\alpha.$ 
\newline
In the special case where $\alpha$ is a positive quadratic irrational, both $\zeta_\alpha$ and $S_\alpha$ have a meromorphic continuation to the whole complex plane with a simple pole at $s=1$ and at most simple poles at $s=-k\pm\dfrac{2\pi i n}{\log\eta_\alpha},$ $(k,n)\in\mathbb{N}\times\mathbb{N}_0.$
\end{theorem}
For the sake of completeness, we will discuss also  the case $\alpha\in(0,1]\cup\mathbb{Q}_+.$ Lastly, we prove that the Hurwitz Zeta-function $\zeta(s;\gamma)$ associated with the Beatty sequence $ \lbrace\lfloor\alpha n\rfloor:n\in\mathbb{N}\rbrace,$ that is
\begin{align*}
\zeta_\alpha(s;\gamma)=\mathlarger\sum\limits_{n=0}^{\infty}\dfrac{1}{(\lfloor\alpha n\rfloor+\gamma)^s},
\end{align*}
where $(\alpha,\gamma)\in(0,+\infty)\times(0,1]$ is fixed and the series converges for $s\in\mathbb{C}_1,$ can also be continued meromorphically beyond the imaginary axis.

\section{Auxiliary Lemmas}
We consider the Banach space 
\begin{align*}\ell^{\infty}:=\left\{\omega=(a_n)_{n\in\mathbb{N}}\in\mathbb{C}^\mathbb{N}:\sup\limits_{n\in\mathbb{N}}|a_n|<\infty\right\},
\end{align*}
endowed with the sup-norm $||\cdot||_\infty$ and the class of ordinary Dirichlet series which are generated from elements of $\ell^\infty$ 
\begin{align*}
\mathcal{D}:=\left\{f_\omega(s):=\mathlarger\sum\limits_{n=1}^{\infty}\dfrac{a_n}{n^s}:\omega=(a_n)_{n\in\mathbb{N}}\in\ell^\infty\right\}.
\end{align*} 
The first lemma, though simple, is of significance for the rest of the results given afterwards. Its proof is inspired by \cite{navas2001analytic}, where the binomial identity
\begin{align*}
(1+z)^{w}=\mathlarger\sum\limits_{m=0}^{\infty}{\binom{w}{ m}}z^m,
\end{align*} 
valid for every $z,w\in\mathbb{C}$ with $|z|<1,$ is used to prove meromorphic continuation for the Fibonacci Dirichlet series.
\begin{lemma}\label{f.l.}If $\mathbb{\omega}=(a_n)_{n\in\mathbb{N}}\in\ell^{\infty},$  $m_0\in\mathbb{N}_0,$ $q\in\mathbb{N}$ and  $j\in\lbrace-1,1\rbrace,$ then  the double series
\begin{align}\label{double}
\mathlarger\sum\limits_{n=q+1}^{\infty}\mathlarger\sum\limits_{m=m_0}^{\infty}(jq)^{m}{\binom{-s}{ m}}\dfrac{a_n}{n^{s+m}}
\end{align}
is absolutely convergent for any $s\in\mathbb{C}_{1-m_0}.$
\end{lemma}

\begin{proof}
Observe that, for any $s\in\mathbb{C},$ any $j\in\lbrace-1,1\rbrace$ and any $m\in\mathbb{N},$
\begin{align*}
\left|j^{m}{\binom{-s}{ m}}\right|&=\left|\dfrac{(-s)(-s-1)\dots(-s-m+1)}{m!}\right|\\&\leq\dfrac{|s|(|s|+1)\dots(|s|+m-1)}{m!}\\
&=(-1)^m\dfrac{(-|s|)(-|s|-1)\dots(-|s|-m+1)}{m!}\\
&=(-1)^m{\binom{-|s|}{ m}}.
\end{align*}
Thus, for any $s\in\mathbb{C}$ with $\sigma:=\text{Re}(s)>1-m_0,$
\begin{align*}
\mathlarger\sum\limits_{n=q+1}^{\infty}&\mathlarger\sum\limits_{m=m_0}^{\infty}\left|(jq)^{m}{\binom{-s}{ m}}\dfrac{a_n}{n^{s+m}}\right|\\
&\leq\mathlarger\sum\limits_{n=q+1}^{\infty}\mathlarger\sum\limits_{m=m_0}^{\infty}(-q)^m{\binom{-|s|}{ m}}\dfrac{|a_n|}{n^{\sigma+m}}\\
&\leq ||\omega||_\infty\mathlarger\sum\limits_{n=q+1}^{\infty}\dfrac{1}{n^{\sigma+m_0}}\mathlarger\sum\limits_{m=m_0}^{\infty}(-q)^m{\binom{-|s|}{ m}}\dfrac{1}{n^{m-m_0}}\\
&\leq ||\omega||_\infty\mathlarger\sum\limits_{n=q+1}^{\infty}\dfrac{1}{n^{\sigma+m_0}}\mathlarger\sum\limits_{m=m_0}^{\infty}(-q)^m{\binom{-|s|}{ m}}\dfrac{1}{(q+1)^{m-m_0}}\\
&\leq ||\omega||_\infty(q+1)^{m_0}\mathlarger\sum\limits_{n=1}^{\infty}\dfrac{1}{n^{\sigma+m_0}}\mathlarger\sum\limits_{m=0}^{\infty}{\binom{-|s|}{ m}}\left(-\dfrac{q}{q+1}\right)^{m}\\
&\leq ||\omega||_\infty(q+1)^{m_0}\zeta(\sigma+m_0)\left(1-\dfrac{q}{q+1}\right)^{-|s|}\\
&=||\omega||_\infty(q+1)^{m_0+|s|}\zeta(\sigma+m_0).
\end{align*}
 Since $\zeta(s)$ converges in $\mathbb{C}_1$, this proves the lemma.
\end{proof}

Let $F,B:\ell^{\infty}\to\ell^{\infty}$ be the forward and the backward shift on $\ell^{\infty}$ 
\begin{align*}
F(a_1,a_2,\dots)=(0,a_1,a_2,\dots)\hspace*{1cm}\text{and}\hspace*{1cm}B(a_1,a_2,\dots)=(a_2,a_3,\dots)
\end{align*} 
for every $\omega=(a_n)_{n\in\mathbb{N}}\in\ell^\infty.$ It is natural to define the corresponding shifts in $\mathcal{D}$ as $\mathcal{F},\mathcal{B}:\mathcal{D}\to\mathcal{D}$ given by
\begin{align*}
\mathcal{F}f_\omega=f_{F\omega}\hspace*{1cm}\text{and}\hspace*{1cm}\mathcal{B}f_\omega=f_{B\omega},
\end{align*}
for every $\omega\in\ell^\infty.$ Obviously, $\mathcal{F}$ and $\mathcal{B}$ are linear maps and, as it will be shown in Lemma \ref{f.t.}, their operation on a Dirichlet series produces another Dirichlet series with the same abscissas as the original one. Before moving on we need to introduce some notation. From this point forward,  $E(f)$ will denote the set of points in $\mathbb{C}_{\sigma_\mu(f)}$ which are poles of $f,$ while  $$E\left(f\right)-m=\left\{s-m:s\in E(f)\right\},$$ for every $m\in\mathbb{N}.$ 

\begin{lemma} \label{f.t.}For every $p,q\in\mathbb{N},$ $b\in\lbrace a,c,\mu\rbrace$ and $\omega\in\ell^{\infty},$ we have that
 \begin{align*}\sigma_b(\mathcal{F}^pf_\omega)=\sigma_b(\mathcal{B}^qf_\omega)=\sigma_b(f_\omega)
 \end{align*}
 and 
 \begin{align*}E(\mathcal{F}^pf_\omega)\cup E(\mathcal{B}^qf_\omega)\subseteq\bigcup\limits_{m=0}^{\infty}\left(E(f_\omega)-m\right).
 \end{align*}
\end{lemma}

\begin{proof}
Let $\omega=(a_n)_{n\in\mathbb{N}}\in\ell^{\infty}$ and $f_\omega\in\mathcal{D}$ be the corresponding Dirichlet series. Since
\begin{align*}\mathcal{B}^qf_\omega(s)=f_{B^q\omega}(s)=\mathlarger\sum\limits_{n=1}^{\infty}\dfrac{a_{n+q}}{n^s}\hspace*{0.5cm}\text{and}\hspace*{0.5cm}\mathcal{F}^pf_\omega(s)=f_{F^p\omega}(s)=\mathlarger\sum\limits_{n=p+1}^{\infty}\dfrac{a_{n-p}}{n^s},
\end{align*}
the Cahen's formulas for the abscissas of ordinary Dirichlet series (see for example in \cite[Chapter IX]{titchmarsh1939theory}) yield
\begin{align*}\sigma_a(\mathcal{F}^pf_\omega)=\sigma_a(\mathcal{B}^qf_\omega)=\sigma_a(f_\omega)\leq1
\hspace*{0.4cm}\text{and}\hspace*{0.4cm}\sigma_c(\mathcal{F}^pf_\omega)=\sigma_c(\mathcal{B}^qf_\omega)=\sigma_c(f_\omega)
.\end{align*}
By applying the binomial identity, one has
\begin{align*}
\mathcal{B}^qf_\omega(s)=\mathlarger\sum\limits_{n=1}^{\infty}\dfrac{a_{n+q}}{(n+q)^s}\left(1-\dfrac{q}{n+q}\right)^{-s}=\mathlarger\sum\limits_{n=q+1}^{\infty}\dfrac{a_n}{n^s}\mathlarger\sum\limits_{m=0}^{\infty}{\binom{-s}{ m}}\left(-\dfrac{q}{n}\right)^m
.\end{align*}
The double series
\begin{align*}
\mathlarger\sum\limits_{n=q+1}^{\infty}\mathlarger\sum\limits_{m=0}^{\infty}(-q)^{m}{\binom{-s}{ m}}\dfrac{a_n}{n^{s+m}}
\end{align*}
is absolutely convergent in $\mathbb{C}_{1},$ by Lemma \ref{f.l.}. Therefore, for every $s\in\mathbb{C}_1,$
\begin{align}\label{B}
\mathcal{B}^qf_\omega(s)
=\mathlarger\sum\limits_{m=0}^{\infty}\mathlarger\sum\limits_{n=q+1}^{\infty}(-q)^{m}{\binom{-s}{ m}}\dfrac{a_n}{n^{s+m}}
=\mathlarger\sum\limits_{m=0}^{\infty}(-q)^{m}{\binom{-s}{ m}}\phi_{q,m}(s),
\end{align}
where
\begin{align*}
\phi_{q,m}(s)=\mathlarger\sum\limits_{n=q+1}^\infty\dfrac{a_n}{n^{s+m}}=\mathcal{F}^q\mathcal{B}^qf_\omega(s+m)
\end{align*}
for every ${(q,m)\in\mathbb{N}\times\mathbb{N}_0}.$ Hence,
\begin{align*}\sigma_a(\phi_{q,m})=\sigma_a(f_\omega)-m\text{\hspace*{0.5cm} and \hspace*{0.5cm}}\sigma_\mu(\phi_{q,m})=\sigma_\mu(f_\omega)-m,
\end{align*}
which implies that $ \phi_{q,m},$ ${(q,m)\in\mathbb{N}\times\mathbb{N}_0},$ are meromorphic functions in $\mathbb{C}_{\sigma_\mu(f_\omega)},$ with 
\begin{align*}E(\phi_{q,m})=E(f_\omega)-m.
\end{align*}
Similarly,
\begin{align}
\mathcal{F}^pf_\omega(s)
&=\mathlarger\sum\limits_{n=p+1}^{2p}\dfrac{a_{n-p}}{n^s}+\mathlarger\sum\limits_{n=2p+1}^{\infty}\dfrac{a_{n-p}}{(n-p)^s}\left(1+\dfrac{p}{n-p}\right)^{-s}\tag*{}\\
&=\mathlarger\sum\limits_{n=p+1}^{2p}\dfrac{a_{n-p}}{n^s}+\mathlarger\sum\limits_{n=p+1}^{\infty}\dfrac{a_{n}}{n^s}\left(1+\dfrac{p}{n}\right)^{-s}\tag*{}\\
\label{F}&=\mathlarger\sum\limits_{n=p+1}^{2p}\dfrac{a_{n-p}}{n^s}+\mathlarger\sum\limits_{m=0}^{\infty}p^m{\binom{-s}{ m}}\phi_{p,m}(s).
\end{align}
Assume that $K$  is a compact subset of $\mathbb{C}_{\sigma_\mu(f)}.$ Then, there exists an $m_0\in\mathbb{N}_0,$ such that
\begin{align*}
1-m_0<\min\limits_{s\in K}\text{Re}(s).
\end{align*}
Moreover, the meromorphic functions $\phi_{q,m},$ $m\geq m_0,$ have no poles in $K$ and the series 
\begin{align*}
\mathlarger\sum\limits_{m=m_0}^{\infty}(jq)^{m}{\binom{-s}{ m}}\phi_{q,m}(s)
\end{align*} 
for $j\in\lbrace-1,1\rbrace$ is uniformly convergent on $K,$ by Lemma \ref{f.l.}.

Thus, $\mathcal{F}^pf_\omega $ and $\mathcal{B}^qf_\omega$ can be continued  meromorphically to $\mathbb{C}_{\sigma_\mu(f_\omega)}$ with
\begin{align*}
E(\mathcal{F}^pf_\omega)\cup E(\mathcal{B}^qf_\omega)\subseteq\bigcup\limits_{m=0}^{\infty} \left(E(\phi_{p,m})\cup E(\phi_{q,m})\right)=\bigcup\limits_{m=0}^{\infty}\left(E(f_\omega)-m\right)
\end{align*}
and 
\begin{align*}\max\lbrace\sigma_\mu(\mathcal{F}^pf_\omega),\sigma_\mu(\mathcal{B}^qf_\omega)\rbrace\leq\sigma_\mu(f_\omega).
\end{align*}
From the preceding inequality and the relation 
\begin{align*}
\mathcal{B}^p\mathcal{F}^pf_\omega(s)=\mathcal{F}^q\mathcal{B}^qf_\omega(s)+\mathlarger\sum\limits_{n=1}^{q}\dfrac{a_n}{n^s}=f_\omega(s),
\end{align*}
the inequality \begin{align*}\sigma_\mu(f_\omega)\leq\min\lbrace\sigma_\mu(\mathcal{F}^pf_\omega),\sigma_\mu(\mathcal{B}^qf_\omega)\rbrace
\end{align*} 
follows and the proof of the lemma is complete.
\end{proof}

\section{Diophantine approximation and Dirichlet series}\label{Diop.appr.}
 Let $R:\mathbb{R}\to\mathbb{R}$ be the odd 1-periodic function 
\begin{align*}
R(x)=\left\{\begin{array}{ll}\lbrace x\rbrace-\dfrac{1}{2}&,x\notin\mathbb{Z}\\
0&,x\in\mathbb{Z}
\end{array}\right.,
\end{align*}
where $\left\{x\right\}$ denotes the fractional part of a real number $x.$

 If $\alpha$ is a positive irrational number, we set ${\boldsymbol\alpha}=\left(R(\alpha n)\right)_{n\in\mathbb{N}}$ and denote by
\begin{align*}
{J}_{\boldsymbol{\alpha}}(s)=\mathlarger\sum\limits_{n=1}^\infty\dfrac{R(\alpha n)}{n^s},
\end{align*}
 which converges for $s\in\mathbb{C}_1,$ the Dirichlet series generated from the sequence ${\boldsymbol\alpha}.$ Let also
\begin{align*}
\alpha:=[a_0,a_1,a_2,\dots],\hspace*{0.3cm}\dfrac{p_n}{q_n}:=[a_0,a_1,a_2,\dots,a_n]\text{\hspace*{0.3cm}and\hspace*{0.3cm}}\theta_n:=[a_{n},a_{n+1},\dots],
\end{align*}
be the continued fraction for $\alpha$, the $n$-th convergent and the $n$-th complete quotient of $\alpha$ , respectively, and define
\begin{align*}
\tau_\alpha:=\limsup\limits_{n\to\infty}\dfrac{\log q_{n+1}}{\log q_n}\in[1,\infty].
\end{align*}
The Dirichlet series $J_{\boldsymbol\alpha}$ was studied extensively in a series of papers by Behnke \cite{behnke1922verteilung}, Hardy and Littlewood \cite{hardy1922some, hardy1924some}, Hecke \cite{hecke1922analytische} and Ostrowski \cite{ostrowski1922bemerkungen}. Its analytic character depends very much on the arithmetical character of $\alpha,$ as the aforementioned authors showed. Indeed it has been proved that:
 
1) for every positive irrational $\alpha,$ we have $\sigma_{c}(J_{\boldsymbol{\alpha}})=1-\dfrac{1}{\tau_\alpha};$ 

2) if $\alpha$ is a positive irrational with $\tau_\alpha>1,$ then $\sigma_{\mu}(J_{\boldsymbol{\alpha}})=1-\dfrac{1}{\tau_\alpha};$ 

3) if $\alpha$ is a positive quadratic irrational, then
\begin{align*}
\sigma_\mu(J_{\boldsymbol\alpha})=-\infty\text{ \hspace*{0.2cm}and\hspace*{0.2cm} }E(J_{\boldsymbol\alpha})\subseteq\left\{-2k\pm\dfrac{2\pi i n}{\log\eta_\alpha}:(k,n)\in\mathbb{N}_0\times\mathbb{N}_0\right\},
\end{align*}
where
\begin{align*}
\eta_\alpha^{-1}:=\theta_{p+1}\theta_{p+2}\dots\theta_{p+q},
\end{align*}  
corresponding to $(a_p,a_{p+1},\dots,a_{p+q-1})$ being the periodic part of the continued fraction of $\alpha.$

By defintion, we know that $\sigma_\mu\leq\sigma_c\leq\sigma_a$. Thus, in order to prove Theorem \ref{gaq} it suffices to show that for any irrational $\alpha>0$ and $q\in\mathbb{Z}^*$,
 \begin{align*}\sigma_{\mu}(g_{\alpha q})=\sigma_{\mu}(J_{\boldsymbol\alpha})-1\text{\hspace*{0.5cm}and\hspace*{0.5cm}} 1\in E(g_{\alpha q})\subseteq\lbrace1\rbrace\cup\left(\bigcup\limits_{m=1}^{\infty}\left(E(J_{\boldsymbol\alpha})-m\right)\right),
 \end{align*}
 where the point $s=1$ will be a simple pole of $g_{\alpha q}$ with residue $\lbrace\alpha q\rbrace.$
 
\begin{proof}[Proof of Theorem \ref{gaq}]
If $\chi_{_S}$ denotes the characteristic function of a set $S,$ we
observe that for every $\alpha\in(0,1]$ and every $x\in(0,1)\setminus\lbrace\alpha\rbrace$ the following equality
\begin{align}\label{saw}
\chi_{_{[0,\alpha)}}(x)=R(x-\alpha)-R(x)+\alpha
\end{align}
holds.

For $q\in\mathbb{N},$ let $I=[0,\lbrace\alpha q\rbrace)$ and $I'=[0,\lbrace-\alpha q\rbrace).$ Then,
\begin{align*}
g_{\alpha q}(s)=\mathlarger\sum\limits_{n\in A_{\alpha q}}\dfrac{1}{n^s}=\mathlarger\sum\limits_{n=1}^{\infty}\dfrac{\chi_{_{I}}(\lbrace\alpha n\rbrace)}{n^s}=\underset{n\neq q}{\mathlarger\sum\limits_{n=1}^{\infty}}\dfrac{\chi_{_{I}}(\lbrace\alpha n\rbrace)}{n^s}
\end{align*}
and
\begin{align*}
g_{-\alpha q}(s)=\mathlarger\sum\limits_{n\in A_{-\alpha q}}\dfrac{1}{n^s}=\mathlarger\sum\limits_{n=1}^{\infty}\dfrac{\chi_{_{I'}}(\lbrace\alpha n\rbrace)}{n^s},
\end{align*}
for every $s\in\mathbb{C}_1.$

Since $\alpha$ is irrational, $\lbrace\alpha n\rbrace\neq\lbrace\alpha q\rbrace$ for every $n\neq q.$ Hence, applying the identity (\ref{saw}) to the $g_{\alpha q}$, gives for every $s\in\mathbb{C}_1$ 
\begin{align}
g_{\alpha q}(s)&=\underset{n\neq q}{\mathlarger\sum\limits_{n=1}^{\infty}}\dfrac{R(\lbrace\alpha n\rbrace-\lbrace\alpha q\rbrace)-R(\lbrace\alpha n\rbrace)+\lbrace\alpha q\rbrace}{n^s}\tag*{}\\
&={\mathlarger\sum\limits_{n=1}^{\infty}}\dfrac{R(\alpha n-\alpha q)}{n^s}-\left(J_{\boldsymbol\alpha}(s)-\dfrac{R(\alpha q)}{q^s}\right)+\lbrace\alpha q\rbrace\left(\zeta(s)-\dfrac{1}{q^s}\right)\tag*{}\\
&=\mathlarger\sum\limits_{n=q+1}^{\infty}\dfrac{R(\alpha(n-q))}{n^s}+\mathlarger\sum\limits_{n=1}^{q}\dfrac{R(\alpha(n-q))}{n^s}-J_{\boldsymbol\alpha}(s)+\lbrace\alpha q\rbrace\zeta(s)-\dfrac{1}{2q^s}\tag*{}\\
\label{Fg1}&=\mathcal{F}^qJ_{\boldsymbol\alpha}(s)-J_{\boldsymbol\alpha}(s)+\lbrace\alpha q\rbrace\zeta(s)+\mathlarger\sum\limits_{n=1}^{q}\dfrac{R(\alpha(n-q))}{n^s}-\dfrac{1}{2q^s}.
\end{align}
By relation (\ref{F}),
\begin{align}
g_{\alpha q}(s)=&\mathlarger\sum\limits_{n=q+1}^{2q}\dfrac{R(\alpha(n-q))}{n^s}+\mathlarger\sum\limits_{{m=1}}^{\infty}q^{m}{\binom{-s}{ m}}\mathcal{F}^qB^qJ_{\boldsymbol\alpha}(s+m)+\mathcal{F}^qB^qJ_{\boldsymbol\alpha}(s)\tag*{}\\
&-J_{\boldsymbol\alpha}(s)+\lbrace\alpha q\rbrace\zeta(s)+\mathlarger\sum\limits_{n=1}^{q}\dfrac{R(\alpha(n-q))}{n^s}-\dfrac{1}{2q^s}\tag*{}\\
\label{Fg}=&\mathlarger\sum\limits_{m=1}^{\infty}q^{m}{\binom{-s}{ m}}\phi_{q,m}(s)+\lbrace\alpha q\rbrace\zeta(s)+h_1(s),
\end{align}
where $h_1$ is an entire function, $\zeta$ has an analytic continuation to the whole complex plane except for a simple pole at $s=1$ wit residue 1, and 
\begin{align*}
\phi_{q,m}(s)=\mathcal{F}^qB^qJ_{\boldsymbol\alpha}(s+m)
\end{align*}
for $m\in\mathbb{N},$  are meromorphic functions defined in $\mathbb{C}_{\sigma_{\mu}(J_{\boldsymbol\alpha})-1}.$ 

Lemma \ref{f.l.} implies that the series of meromorphic functions in (\ref{Fg}) is uniformly convergent in compact subsets of $\mathbb{C}_{\sigma_{\mu}(J_{\boldsymbol\alpha})-1}.$
Thus, \begin{align*}\sigma_\mu(g_{\alpha q})\leq\max\left\{\sigma_{\mu}(\phi_{q,m}):m\geq 1\right\}=\sigma_\mu(J_{\boldsymbol\alpha})-1.
\end{align*}

  and 
  \begin{align*}1\in E(g_{\alpha q})&\subseteq E(\zeta)\cup \left(\bigcup\limits_{m=1}^{\infty}E(\phi_{q,m})\right)\subseteq\lbrace1\rbrace\cup\left(\bigcup\limits_{m=1}^{\infty}\left(E(J_{\boldsymbol\alpha})-m)\right)\right),
  \end{align*} 
  where the point $s=1$ is a simple pole of $g_{\alpha q}$ with residue $\lbrace\alpha q\rbrace.$
  
  Now assume that $\sigma_\mu(g_{\alpha q})<\sigma_\mu(J_{\boldsymbol\alpha})-1.$  From relation (\ref{Fg}) 
  \begin{align*}
  sq\phi_{q,1}(s)=-g_{\alpha q}(s)+\mathlarger\sum\limits_{m=2}^{\infty}q^{m}{\binom{-s}{ m}}\phi_{q,m}(s)+\lbrace\alpha q\rbrace\zeta(s)+h_1(s),
  \end{align*}
  where the series of meromorphic functions is uniformly convergent in compact subsets of $\mathbb{C}_{\sigma_{\mu}(J_{\boldsymbol\alpha})-2}.$ Hence,
  \begin{align*}
  \sigma_{\mu}(J_{\boldsymbol\alpha})-1=\sigma_{\mu}(sq\phi_{q,1})\leq\max\left\{\sigma_{\mu}(g_{\alpha q}),\sigma_{\mu}(\phi_{q,m}):m\geq 2\right\}\leq\sigma_{\mu}(J_{\boldsymbol\alpha})-2,
  \end{align*}
  which is impossible. Therefore, $\sigma_\mu(g_{\alpha q})=\sigma_\mu(J_{\boldsymbol\alpha})-1.$
  
It is left to consider the case of $g_{-\alpha q}$. Since $\lbrace-x\rbrace=1-\lbrace x\rbrace$ for all $x>0,$ and $\lbrace\alpha n\rbrace\neq1-\lbrace\alpha q\rbrace$ for every $n\in\mathbb{N},$ one has
 \begin{align*}
 g_{-\alpha q}(s)&={\mathlarger\sum\limits_{n=1}^{\infty}}\dfrac{R(\lbrace\alpha n\rbrace-(1-\lbrace\alpha q\rbrace))-R(\lbrace\alpha n\rbrace)+1-\lbrace\alpha q\rbrace}{n^s}\\
 &=\mathlarger\sum\limits_{n=1}^\infty\dfrac{R(\alpha(n+q))}{n^s}-J_{\boldsymbol\alpha}(s)+(1-\lbrace\alpha q\rbrace)\zeta(s)\tag*{},
 \end{align*}
 or
 \begin{align}\label{Bg}
g_{-\alpha q}(s)
&=\mathcal{B}^qJ_{\boldsymbol\alpha}(s)-J_{\boldsymbol\alpha}(s)+(1-\lbrace\alpha q\rbrace)\zeta(s),
 \end{align}
which is valid for every $s\in\mathbb{C}_1.$ As before, using this time relation (\ref{B}), we have
 \begin{align*}\sigma_\mu(g_{-\alpha q})=\sigma_\mu(J_{\boldsymbol\alpha})-1
\hspace*{0.5cm}\text{and}\hspace*{0.5cm}
1\in E(g_{-\alpha q})\subseteq\lbrace1\rbrace\cup\left(\bigcup\limits_{m=1}^{\infty}(E(J_{\boldsymbol\alpha})-m)\right),
\end{align*}
 where the point $s=1$ is a simple pole of $g_{-\alpha q}$ with residue $\lbrace-\alpha q\rbrace.$
  \end{proof}
  
\section{Beatty Zeta-functions and Sturmian Dirichlet series}

We now turn our focus on the Beatty Zeta-functions and Sturmian Dirichlet serires. 
The summatory functions of $\zeta_\alpha$ and $S_\alpha$ are given, respectively, by 
\begin{align*}
G_1(x)=\underset{n\in B(\alpha)}{\mathlarger\sum\limits_{n\leq x}}1=\left\lfloor\left(\lfloor x\rfloor+1\right)\beta\right\rfloor
\end{align*}
and
\begin{align*}G_2(x)=\mathlarger\sum\limits_{n\leq x}(\left\lceil\beta n\right\rceil-\left\lceil\beta(n-1)\right\rceil)=\left\lceil\beta \lfloor x\rfloor\right\rceil-1,
\end{align*}
 for every $x\in\mathbb{R}_+,$ where $\beta$ is the reciprocal of $\alpha.$ In addition, the terms that appear in those sums are non-negative. Thus, the Cahen's formulas for the abscissas of ordinary Dirichlet series  yield
 \begin{align*}
 \sigma_a(\zeta_\alpha)=\sigma_a(S_\alpha)= \sigma_c(\zeta_\alpha)=\sigma_c(S_\alpha)=1.
 \end{align*}
One of the techniques that can be used in order to achieve a meromorphic continuation of $\zeta_\alpha$ and $S_\alpha$ beyond the vertical line $1+i\mathbb{R},$ is Abel's summation formula. Indeed, after elementary computations, one has
\begin{align*}
\zeta_\alpha(s)=\beta\lfloor\alpha\rfloor^{1-s}\dfrac{s}{s-1}+s\mathlarger\int_{\lfloor\alpha\rfloor}^{\infty}\dfrac{G_1(u)-\beta u}{u^{s+1}}\mathrm{d}u
\end{align*} 
and
\begin{align*}
S_\alpha(s)=\beta\dfrac{s}{s-1}+s\mathlarger\int_{1}^{\infty}\dfrac{G_2(u)-\beta u}{u^{s+1}}\mathrm{d}u,
\end{align*}
valid for every $s\in\mathbb{C}_1.$ However, the integral expressions appearing in the above identities are uniformly convergent in half-planes $\mathbb{C}_\delta,$ for every $\delta>0.$ Therefore,
 both $\zeta_\alpha$ and $S_\alpha,$ can be continued analytically to $\mathbb{C}_0$ except for a simple pole at the point $s=1$ with residue $\beta.$ 

The question that arises is whether $\zeta_\alpha$ and $S_\alpha$  can be continued meromorphically even further to the left, beyond the imaginary axis, which is something that can not be obtained with Abel's summation formula anymore. Intuitively, it should be expected that, as in the case of $J_{\boldsymbol\alpha},$ the arithmetical character of $\alpha$ will play a prominent role to the answer of this question. This is what Theorem \ref{BS} states for $\alpha>1$ irrational number.

\begin{proof}[Proof of Theorem \ref{BS} ]Let $\alpha>1$ be irrational and $\alpha'$ be defined by
\begin{align*}
\dfrac{1}{\alpha}+\dfrac{1}{\alpha'}=1.
\end{align*}
The well-known Beatty's (or in some literature Rayleigh's) theorem states that the sets $\lbrace\lfloor\alpha n\rfloor:n\in\mathbb{N}\rbrace$ and  $\lbrace\lfloor\alpha' n\rfloor:n\in\mathbb{N}\rbrace$ form a partition of $\mathbb{N}.$ Observe that, for every $n\in\mathbb{N},$
\begin{align*}
\left\{\dfrac{\lfloor\alpha n\rfloor+1}{\alpha}\right\}=\left\{n+\dfrac{1-\lbrace\alpha n\rbrace}{\alpha}\right\}\in\left(0,\dfrac{1}{\alpha}\right)
\end{align*}
and
\begin{align*}
\left\{\dfrac{\lfloor\alpha' n\rfloor+1}{\alpha}\right\}=\left\{(\lfloor\alpha' n\rfloor+1)\left(1-\dfrac{1}{\alpha'}\right)\right\}\in\left(\dfrac{1}{\alpha},1\right).
\end{align*}
Thus, we obtain from (\ref{Hecke}),
\begin{align*}
\mathcal{F}\zeta_\alpha(s)=\mathlarger\sum\limits_{n=1}^\infty\dfrac{1}{(\lfloor\alpha n\rfloor+1)^s}=\mathlarger\sum\limits_{n\in A_{\beta}}\dfrac{1}{n^s}=g_{_\beta}(s)
\end{align*}
for every $s\in\mathbb{C}_1,$ or, by (\ref{Fg1}) and (\ref{Bg}),
\begin{align}\label{Fz}\zeta_\alpha=\mathcal{B}g_{_\beta}=\mathcal{B}\left(\mathcal{F}J_{\boldsymbol{\beta}}-J_{\boldsymbol{\beta}}+{\beta}\zeta-\dfrac{1}{2}\right)=J_{\boldsymbol{\beta}}-\mathcal{B}J_{\boldsymbol{\beta}}+\beta\zeta=\zeta-g_{_{-\beta}}.
\end{align}
Since $\tau_\alpha=\tau_{_\beta}$, Theorem \ref{gaq} and relation (\ref{Fz}) yield Theorem \ref{BS} for $\zeta_\alpha.$
 
In the case of Sturmian Dirichlet series,  it follows from (\ref{Bg}) and (\ref{Fz}) that, for every $s\in\mathbb{C}_1,$ 
\begin{align}
\mathcal{B}S_{\alpha}(s)&=\mathlarger\sum\limits_{n=1}^{\infty}\dfrac{\left\lceil\beta(n+1)\right\rceil-\left\lceil\beta  n\right\rceil}{n^s}\tag*{}\\
&=\mathlarger\sum\limits_{n=1}^{\infty}\dfrac{\beta(n+1)+1-\left\{\beta(n+1)\right\}-\beta n-1+\left\{\beta n\right\}}{n^s}\tag*{}\\
&=\mathlarger\sum\limits_{n=1}^\infty\dfrac{1}{n^s}-\mathlarger\sum\limits_{n=1}^\infty\dfrac{R\left(\beta(n+1)\right)-R\left(\beta n\right)+1-\beta}{n^s}\tag*{}\\
&=\zeta(s)-g_{_{-\beta}}(s)\tag*{}\\
&=\zeta_\alpha(s)\tag*{},
\end{align}
or
\begin{align}\label{subnice}
S_\alpha-1=S_\alpha-\left\lceil\beta\right\rceil=\mathcal{F}\mathcal{B}S_\alpha=\mathcal{F}\zeta_\alpha=g_{_\beta}.
\end{align}
Once more, Theorem \ref{gaq} and relation (\ref{subnice}) will yield Theorem \ref{BS} for $S_\alpha.$
\end{proof}

\section{The case of $\alpha\in(0,1]\cup\mathbb{Q}_+$}
In the case of a rational parameter $\alpha=\frac{p}{q}\geq1,$ where $p$ and $q$ are coprime natural numbers, the coefficients of $\zeta_\alpha$ and $S_\alpha$ show a periodic behaviour and their meromorphic continuation to the whole complex plane follows from the one of the Hurwitz Zeta-function
\begin{align*}
\zeta(s;\gamma):=\mathlarger\sum\limits_{n=0}^{\infty}\dfrac{1}{(n+\gamma)^s},
\end{align*}
where $\gamma\in(0,1]$ is a fixed parameter and the series converges for $s\in\mathbb{C}_1.$ Indeed, the Hurwitz Zeta-function can be continued analytically to the whole complex plane except for a simple pole at $s=1$ with residue 1, while the absolute convergence of $\zeta_\alpha$ and $S_\alpha$ in $\mathbb{C}_1$ allows us to rearrange the terms of the respective series
such that
\begin{align*}
\zeta_\alpha(s)\hspace*{-0.05cm}=\hspace*{-0.05cm}\mathlarger\sum\limits_{\ell=1}^{q}\mathlarger\sum\limits_{n=0}^{\infty}
\dfrac{1}{\left\lfloor \frac{p}{q}(nq+\ell)\right\rfloor^s}\hspace*{-0.05cm}=\hspace*{-0.05cm}\mathlarger\sum\limits_{\ell=1}^{q}\mathlarger\sum\limits_{n=0}^{\infty}
\frac{1}{\left( np+\left\lfloor\frac{\ell p}{q}\right\rfloor\right)^s}\hspace*{-0.05cm}=\hspace*{-0.05cm}\frac{1}{p^s}{\mathlarger\sum\limits_{\ell=1}^{q}}\zeta\left(s;\frac{1}{p}\left\lfloor\frac{\ell p}{q}\right\rfloor\right)
\end{align*}
and
\begin{align*}
S_\alpha(s)&=\mathlarger\sum\limits_{\ell=1}^{p}\mathlarger\sum\limits_{n=0}^{\infty}
\dfrac{\left\lceil\frac{q}{p}(np+\ell)\right\rceil-\left\lceil\frac{q}{p}(np+\ell-1)\right\rceil}{(np+\ell)^s}\\
&=\mathlarger\sum\limits_{\ell=1}^{p}\mathlarger\sum\limits_{n=0}^{\infty}
\dfrac{\left\lceil\frac{\ell q}{p}\right\rceil-\left\lceil\frac{(\ell-1)q}{p}\right\rceil}{(np+\ell)^s}\\
&=\dfrac{1}{p^s}\mathlarger\sum\limits_{k=0}^{q-1}\mathlarger\sum\limits_{\ell=\left\lfloor\frac{kp}{q}\right\rfloor+1}^{\left\lfloor\frac{(k+1)p}{q}\right\rfloor}\left(\left\lceil\frac{\ell q}{p}\right\rceil-\left\lceil\frac{(\ell-1)q}{p}\right\rceil\right)\zeta\left(s;\dfrac{\ell}{p}\right)\\
&=\dfrac{1}{p^s}\mathlarger\sum\limits_{k=1}^{q}\zeta\left(s;\dfrac{1}{p}\left(\left\lfloor\frac{kp}{q}\right\rfloor+1\right)\right).
\end{align*}

The case of $0<\alpha<1$ is contained in a sense in the case of $\alpha>1.$ Indeed, let $0<\alpha<1$ and $\beta$ the reciprocal $\alpha.$ 
The corresponding Sturmian Dirichlet series can be expressed, for $s\in\mathbb{C}_1,$ as
\begin{align*}
\mathcal{B}S_\alpha(s)&=\mathlarger\sum\limits_{n=1}^{\infty}
\dfrac{\left\lceil\beta(n+1)\right\rceil-\left\lceil\beta n\right\rceil}{n^s}\\
&=\mathlarger\sum\limits_{n=1}^{\infty}\dfrac{\left\lceil\left\{\beta\right\} (n+1)\right\rceil-\left\lceil\left\{\beta\right\} n\right\rceil+\left\lfloor\beta\right\rfloor}{n^s}\\
&=\left\lfloor\beta\right\rfloor\zeta(s)+\mathcal{B}S_{\left\{\beta\right\}^{-1}}(s).
\end{align*}
Applying the $\mathcal{F}$ operator on both sides, yields eventually \begin{align*}S_{\alpha}=\left\lfloor\beta\right\rfloor\zeta+S_{\left\{\beta\right\}^{-1}}
.\end{align*}
On the other hand, the running index of the sum in the Dirichlet series $\zeta_\alpha$ has to start from $\left\lfloor\beta\right\rfloor+1$ such that everything is well defined. Notice, however, that for $s\in\mathbb{C}_1,$ 
\begin{align*}
\zeta_\alpha(s)=\mathlarger\sum\limits_{n=\left\lfloor\beta\right\rfloor+1}^\infty\dfrac{1}{\left\lfloor\alpha n\right\rfloor^s}=\mathlarger\sum\limits_{n=1}^\infty\dfrac{a_n}{n^s},
\end{align*}
where $a_n=\#\left\{m\in\mathbb{N}:n=\lfloor\alpha m\rfloor\right\}
=\left\lceil\beta (n+1)\right\rceil-\left\lceil\beta n\right\rceil
.$ Thus, we have, as in the case of $\alpha>1,$ that \begin{align*}\zeta_\alpha=\mathcal{B}S_{\alpha}
=\left\lfloor\beta\right\rfloor\zeta+\mathcal{B}S_{\left\{\beta\right\}^{-1}}
=\left\lfloor\beta\right\rfloor\zeta+\zeta_{\left\{\beta\right\}^{-1}}
.\end{align*}
Since, $\lbrace\beta\rbrace^{-1}>1,$ we can include the case of $0<\alpha<1$ to the results we have proved so far.

\section{Hurwitz Zeta-functions associated with Beatty sequences}
We conclude with a generalization of our previous results regarding the Beatty Zeta-function, which can be compared with the ones in  \cite{banks2017certain}.

If $(\alpha,\gamma)\in(0,+\infty)\times(0,1],$ is a fixed pair of real numbers, denote by
\begin{align*}
\zeta_\alpha(s;\gamma)=\mathlarger\sum\limits_{n=0}^{\infty}\dfrac{1}{(\lfloor\alpha n\rfloor+\gamma)^s},
\end{align*}
converging for $s\in\mathbb{C}_1,$ the Hurwitz Zeta-function $\zeta(s;\gamma)$ associated with the Beatty sequence $B(\alpha).$ In the case of $\gamma=1,$ one has $\zeta_\alpha(\cdot\ ;1)={{\mathcal{F}\zeta_\alpha+1}}.$

If $0<\gamma<1$ then, for every $s\in\mathbb{C}_1,$
\begin{align*}
\zeta_\alpha(s;\gamma)=\mathlarger\sum\limits_{n=0}^{\infty}\dfrac{1}{(\lfloor\alpha n\rfloor+\gamma)^s}=\gamma^{-s}+\mathop{\mathlarger\sum\limits_{n=1}^{\infty}}_{n\in B(\alpha)}\dfrac{1}{(n+\gamma)^s}
\end{align*}
or, using the binomial identity,
\begin{align*}
\zeta_\alpha(s;\gamma)-\gamma^s=\mathop{\mathlarger\sum\limits_{n=1}^{\infty}}_{n\in B(\alpha)}\dfrac{1}{n^s}\left( 1+\dfrac{\gamma}{n}\right)^{-s}=\mathop{\mathlarger\sum\limits_{n=1}^{\infty}}_{n\in B(\alpha)}\dfrac{1}{n^s}\mathlarger\sum\limits_{m=0}^{\infty}{\binom{-s}{ m}}\left(\dfrac{\gamma}{n}\right)^m.
\end{align*}
A variant of Lemma \ref{f.l.} yields, for every $s\in\mathbb{C}_1,$
\begin{align*}
\zeta_\alpha(s;\gamma)-\gamma^{-s}&=\mathlarger\sum\limits_{m=0}^{\infty}{\binom{-s}{ m}}\gamma^m\mathop{\mathlarger\sum\limits_{n=1}^{\infty}}_{n\in B(\alpha)}\dfrac{1}{n^{s+m}}=\mathlarger\sum\limits_{m=0}^{\infty}{\binom{-s}{ m}}\gamma^m\zeta_\alpha(s+m).
\end{align*}
Therefore, with minor modifications in the arguments discussed in the previous sections, one obtains 

\begin{corollary*} If $(\alpha,\gamma)\in(0,+\infty)\times(0,1),$ then
\begin{align*}
\sigma_{\mu}(\zeta_{\alpha}(\cdot\ ;\gamma))=\sigma_{\mu}(\zeta_{\alpha})
\end{align*}
and
\begin{align*}E(\zeta_{\alpha}(\cdot\ ;\gamma))\subseteq\bigcup\limits_{m=1}^{\infty}(E(\zeta_{\alpha})-m).
\end{align*}
For every $(\alpha,\gamma)\in((0,+\infty)\cap\mathbb{Q})\times(0,1],$ the function $\zeta_\alpha(s;\gamma)$ has a meromorphic continuation to the whole complex plane with at most simple poles at the points $s=1-m,$ $m\in\mathbb{N}_0,$ with residue 
\begin{align*}\beta{\binom{1-m}{ m}}\gamma^m=\left\{\begin{array}{ll}\beta&,m=0\\
0&,m\geq1
\end{array}\right.,
\end{align*}
respectively. Hence, $\zeta_\alpha(\cdot\ ;\gamma)$ has only a simple pole at $s=1.$

For every $\left(\alpha,\gamma)\in((0,+\infty)\setminus\mathbb{Q}\right)\times(0,1],$ the function $\zeta_\alpha(s;\gamma)$ can be continued analytically to $\mathbb{C}_{-\frac{1}{\tau_{\alpha}}}$ except for a simple pole at $s=1$ with residue $\beta.$ If $\tau_\alpha>1,$ the vertical line $-\dfrac{1}{\tau_\alpha}+i\mathbb{R}$ is the natural boundary for $\zeta_\alpha(s;\gamma).$ In the special case where $\alpha$ is a positive quadratic irrational, $\zeta_\alpha(s;\gamma)$ has a meromorphic continuation to the whole complex plane with a simple pole at $s=1$ and at most simple poles at $s=-k\pm\dfrac{2\pi i n}{\log\eta_\alpha},$ $(k,n)\in\mathbb{N}\times\mathbb{N}_0.$
\end{corollary*}
\section{Concluding Remarks}
For $\alpha>0$ we have seen that 
\begin{align*}\mathlarger\sum\limits_{n=1}^{\infty}\dfrac{\chi_{_{B(\alpha)}}(n)}{n^s}=\zeta_\alpha(s)=\mathcal{B}S_\alpha(s)=\mathlarger\sum\limits_{n=1}^{\infty}\dfrac{\left\lceil\beta(n+1)\right\rceil-\left\lceil\beta n\right\rceil}{n^s},
\end{align*}
for every $s\in\mathbb{C}_1.$ The uniqueness theorem for ordinary Dirichlet series implies that the corresponding coefficients of $\zeta_\alpha$ and $\mathcal{B}S_\alpha$ are identical, or
\begin{align*}
\left\lceil\beta(n+1)\right\rceil-\left\lceil\beta n\right\rceil=\chi_{_{B(\alpha)}}(n)=\left\{\begin{array}{ll}1&,n\in B(\alpha)\\
0&,n\notin B(\alpha)
\end{array}
\right..
\end{align*}

The relation of $\zeta_\alpha$ and $S_\alpha$, which at first sight is not noticeable, explains the similarities of the results that can be found about them in different sources of the literature, seemingly unrelated, like for example in  \cite{kwon2015one} and \cite{o2002generating}. It would be interesting to consider which properties for one Dirichlet series apply also to the other one, something that will not be pursued here though.

So far we have proved which points of the complex plane could be possibly poles for $\zeta_\alpha$ and $S_\alpha,$ in case of a quadratic irrational $\alpha.$ Only the point $s=1$ seems to be definitely a pole. However, Hecke's method \cite{hecke1922analytische} allows to specifically locate all poles of $J_{\boldsymbol\alpha}.$ Thus, in view of relations (\ref{Fg}), (\ref{Bg}), (\ref{Fz}) and (\ref{subnice}) the same can be done for $\zeta_\alpha$ and $S_\alpha.$

Quite often the irrationals $\alpha$ with $\tau_\alpha=\tau<\infty$ are called $\tau$-Diophantine numbers. The set of $1$-Diophantine numbers has full Lebesgue measure in $\mathbb{R}.$ It follows from the celebrated theorem of Roth, that the algebraic irrationals are among these numbers. On the other hand, the irrationals $\alpha$ with $\tau_\alpha=\infty$ are called Liouville numbers and their set is the complement of the set of 1-Diophantine numbers, whence of zero measure.
The reader interested on properties of such numbers, may look at \cite{lang1995introduction} or \cite{steuding2005diophantine}. 

Returning to $J_\alpha,$ $\zeta_\alpha$ and $S_\alpha,$ we have already seen that if $\tau_\alpha>1,$ then the vertical line $1-\dfrac{1}{\tau_\alpha}+i\mathbb{R}$ is the natural boundary of $J_{\boldsymbol{\alpha}}$ and $-\dfrac{1}{\tau_\alpha}+i\mathbb{R}$ of $\zeta_\alpha$ and $S_\alpha.$ Hardy and Littlewood \cite{hardy1922some} conjectured that this is also the case for $J_{\boldsymbol\alpha}$ if $\alpha$ is 1-Diophantine number but not quadratic irrational, for which a meromorphic continuation to the whole complex plain has been given.  The conjecture is still open.


\begin{thebibliography}{10}

\bibitem{banks2017certain}
{ W.D. Banks}, {\em On certain zeta functions associated with beatty
  sequences}, arXiv preprint arXiv:1705.09969,  (2017).

\bibitem{behnke1922verteilung}
{ H.~Behnke}, {\em {{\"U}ber die Verteilung von Irrationalit{\"a}ten mod. 1}},
  in Abhandlungen aus dem Mathematischen Seminar der Universit{\"a}t Hamburg,
  vol.~\textbf{1}, Springer, 1922, 251--266.

\bibitem{hardy1922some}
{ G.H. Hardy, J.E. Littlewood}, {\em {Some problems of Diophantine
  approximation: The lattice-points of a right-angled triangle.(Second
  memoir.)}}, in Abhandlungen aus dem Mathematischen Seminar der
  Universit{\"a}t Hamburg, vol.~\textbf{1}, Springer, 1922, 211--248.

\bibitem{hardy1924some}
\leavevmode\vrule height 2pt depth -1.6pt width 23pt, {\em {Some problems of
  diophantine approximation: the analytic character of the sum of a
  Dirichlet’s series considered by Hecke}}, in Abhandlungen aus dem
  Mathematischen Seminar der Universit{\"a}t Hamburg, vol.~\textbf{3},
  Springer, 1924, 57--68.

\bibitem{hecke1922analytische}
{ E.~Hecke}, {\em {{\"U}ber analytische Funktionen und die Verteilung von
  Zahlen mod. eins}}, in Abhandlungen aus dem Mathematischen Seminar der
  Universit{\"a}t Hamburg, vol.~\textbf{1}, Springer, 1922, 54--76.

\bibitem{kwon2015one}
{ DoYong Kwon}, {\em {A one-parameter family of Dirichlet series whose
  coefficients are Sturmian words}}, Journal of Number Theory, \textbf{147}
  (2015), 824--835.

\bibitem{lang1995introduction}
{ S.~Lang}, {\em Introduction to diophantine approximations}, Springer Science
  \& Business Media, 1995.

\bibitem{navas2001analytic}
{ Luis Navas}, {\em {Analytic continuation of the Fibonacci Dirichlet series}},
  Fibonacci Quarterly, \textbf{39} (2001), 409--418.

\bibitem{o2002generating}
{ Kevin O'Bryant}, {\em {A generating function technique for Beatty sequences
  and other step sequences}}, Journal of Number Theory, \textbf{94} (2002),
  299--319.

\bibitem{ostrowski1922bemerkungen}
{ A.~Ostrowski}, {\em {Bemerkungen zur Theorie der diophantischen
  Approximationen}}, in Abhandlungen aus dem Mathematischen Seminar der
  Universit{\"a}t Hamburg, vol.~\textbf{1}, Springer, 1922, 77--98.

\bibitem{steuding2005diophantine}
{ J.~Steuding}, {\em Diophantine analysis}, Springer, 2005.

\bibitem{titchmarsh1939theory}
{ E.C. Titchmarsh}, {\em The theory of functions},  (1939).

\end{thebibliography}

\end{document}